\documentclass[12pt]{article}
\usepackage{amsmath,amssymb, amsfonts, amsthm,amscd}
\usepackage[T2A]{fontenc}
\usepackage[cp1251]{inputenc}
\usepackage{graphicx}

\theoremstyle{plain}
\newtheorem{theorem}{Theorem}

\newtheorem{proposition}{Proposition}

\newtheorem{definition}{Definition}
\theoremstyle{definition}

\newcommand{\QR}{Q_{Cl}(R)}
\newcommand{\UQR}{U(Q_{Cl}(R))}
\newcommand{\BQR}{\mathfrak{B}(Q_{Cl}(R))}
\newcommand{\RR}{\mathfrak{R}(R)}
\newcommand{\UR}{U(R)}
\newcommand{\BR}{\mathfrak{B}(R)}
\newcommand{\JR}{J(R)}

\newcommand{\Z}{\mathbb{Z}}
\newcommand{\de}{\delta}
\renewcommand{\phi}{\varphi}

\newcommand{\sbs}{\subset}

\begin{document}
\begin{center}
{\huge Type conditions of stable range for identification of qualitative generalized classes of rings}
\end{center}
\vskip 0.1cm \centerline{{\Large Bohdan Zabavsky}}

\vskip 0.3cm

\centerline{\footnotesize{Department of Mechanics and Mathematics,  Ivan Franko National University of Lviv,  Ukraine}}
\vskip 0.5cm

\centerline{\footnotesize{March, 2016}}
\vskip 0.7cm

\footnotesize{\noindent\textbf{Abstract:} \textit{
      This article deals mostly with the following question: when is the classical ring of quotients of a commutative ring a ring of stable range 1?
      We introduce the concepts of a ring of (von Neumann) regular range 1, a ring of semihereditary range 1, a ring of regular range 1, a semihereditary local ring, a regular local ring. We find relationships between the introduced classes of rings and known ones, in particular, it is established  that a commutative indecomposable almost clean ring is a regular local ring. A commutative any ring of idempotent regular range 1 is an almost clean ring. It is shown that any commutative indecomposable almost clean Bezout ring is an Hermite ring, any commutative semihereditary ring is a ring of idempotent regular range 1. The classical ring of quotients of a commutative Bezout ring $\QR$ is a (von Neumann) regular local ring if and only if $R$ is a commutative semihereditary local ring.}
}

\vskip 1cm

\normalsize

\section{Introduction}

\subsection{Terminology and notion}

Throughout, all rings are assumed to be associative with unit and $1 \neq 0$. The set of nonzero divisors (also called regular elements) of $R$ is denoted by $\RR$, the set of units by $\UR$  and the set of idempotents by $\BR$. The Jacobson radical of a ring $R$ is denoted by $\JR$. The classical ring of quotients of ring $R$ is denoted by $\QR$.

A ring $R$ is called indecomposable if $\BR = \{0, 1\}$. A ring is called clean if every its element is the sum of a unit and idempotents, and it it called almost clean if each its element is the sum of a regular element and an idempotent~\cite{McG}. An element $a$ of a  ring $R$ is called (von Neumann) regular element, if $axa = a$ for some element $x \in R$. An element $a$ of a ring $R$ is called a left (right) semihereditary element if $Ra$ ($aR$) is projective. A ring $R$ is {\it a ring of stable range 1}, if for any $a, b \in R$ such that $aR+bR=R$ there exists $t\in R$ such that $(a+bt)R = R$. A ring $R$ is {\it a ring of stable range 2}, if for any $a, b, c \in R$ such that $aR + bR + cR = R$ there exist $x, y \in R$ such that $(a+cx)R + (b+cy)R= R$ (see \cite{Zabavsk}).

Following Kaplansky \cite{Kapl} a commutative ring is said to be an elementary divisor ring if every  matrix $A$ over $R$ is equivalent to a diagonal matrix, i.e. for $A$ there exist invertible matrices $P$ and $Q$ of appropriate sizes such that $PAQ$ is diagonal matrix $(d_{ij})$  (i.e $d_{ij} = 0$ whenever $i \neq j$) with the property that $Rd_{i+1,i+1}R\subseteq d_{ii}R\cap Rd_{ii}$. If every 1 by 2 and 2 by 1 matrix over $R$ is equivalent to a diagonal matrix then the ring called an Hermite ring.

Obviously, an elementary divisor ring is Hermite and it is easy to see that an Hermite ring is Bezout \cite{McG}. Examples that neither implication is revertible are provided by Gillmann and Henriksen in \cite{GH}. We have the following result.

\begin{theorem}\cite{Zabavsk}\label{t1.1.1}
      A commutative Bezout ring $R$ is an Hermite ring if and only if the stable range of $R$ is equal 2.
\end{theorem}

Contessa in \cite{Cont} called a ring a (von Neumann) regular local ring if for each $a \in R$ either $a$ or $1 - a$ is a (von Neumann) regular element.

\subsection{Useful facts known}

\begin{proposition}\label{p1.2.1}
      Let $R$ be a commutative Bezout ring. If $\phi \in \BQR$ then $\phi \in \BR$.
\end{proposition}

\begin{proof}
      Let $\phi \in \BQR$ and $\phi = \frac{e}{s}$, where $s$ is a regular element of $R$. Let $eR + sR = \de R$, then $e = e_0\de$, $s = s_0\de$ and $eu + sv = \de$ for some elements $e_0, s_0, u, v \in R$. Since $s$ is a regular element,  $\de$ is a regular element as a divisor of $s$. Since  $eu + sv = \de$, then  $\de(e_0u + s_0v - 1) = 0$. Since $\de \neq 0$ and $\de$ is a regular element of $R$, we have $e_0u + s_0v - 1 = 0$. Then $\frac{e}{s}=\frac{e_0}{s_0}$, where $e_0R + s_0R = R$. Since $\frac{e_0}{s_0} \in \BQR$, then $e_0^2s_0 = e_0s_0^2$ and $s_0(e_0^2 - e_0s_0) = 0$. Since $s_0 \neq 0$ and so $s_0$ is a regular element of $R$ as a divisor of $s$, we have $e_0^2=e_0s_0$.

      Since $e_0u + s_0v = 1$, we have $e_0^2u + e_0s_0v = e_0$ and $s_0(e_0u + s_0v) = e_0$. Hence $\frac{e_0}{s_0} \in R$.
\end{proof}

\begin{proposition}\label{p1.2.2}
      Let $R$ be a commutative ring and $a$ is a (von Neumann) regular element of $R$. Then $a = eu$, where $e \in \BR$ and $u \in \UR$.
\end{proposition}

\begin{proof}
      Let $axa = a$. This implies $axax = ax$, i.e. $e = ax \in \BR$ and $e \in aR$. Since $axa = a$, then $ea = a$, i.e. $a \in eR$ and we have $aR = eR$.

      Consider the element $u = (1 - e) + a$. Since $u(1 - e) = 1 - e$, we have $uR + eR = R$. We proved that $eR = aR$, then $uR + aR = R$. Since $ue = ((1 - e) + a)e = ae = a$, then $aR \sbs uR$. Obviously, the  equality $uR + aR = R$ and inclusion $aR \sbs uR$ in a commutative ring is possible if $u \in \UR$.

      Then we have $ue = a$.
\end{proof}

\begin{proposition}\label{p1.2.3}
      Let $R$ be a commutative ring and let $a$ be a semihereditary element if and only if $a = er$, where $e \in \BR$ and $r \in \RR$.
\end{proposition}

\begin{proof}
      Let $\phi R = \{\, x \mid xa = 0\}$ and $\phi \in \BR$. Since $\phi a = 0$, we have $(1 - \phi)a=a$. Let $r = a - \phi$ and $rx = 0$.

      Since $ax = \phi x$ and $(1 - \phi)a = a$, we have $(1 - \phi)ax = \varphi x$ and $(1 - \phi)\phi x = 0$. Then $\phi x = 0$ and $a x = 0$. Since $ax = 0$, we have $x \in \phi R$, i.e. $x = x\phi$. Since $x\phi = 0$, we have $x = 0$. Then we see that $r$ is a regular element of $R$. Since
      $$r(1 - \phi) = a(1 - \phi) - \phi(1 - \phi) = a(1 - \phi) = a,$$
      i.e. $a = r(1 - \phi)$. Put $1 - \phi = e$, we have $a = re$, where $e \in \BR$ and $r \in \RR$. Obviously, $\{x | x(re) = 0\} = (1 - e)R$.
\end{proof}

\section{Range conditions on the rings}

\begin{definition}\label{d2.1.1}
      A ring $R$ is said to have a (von Neumann) regular range 1, if for any $a, b \in R$ such that $aR + bR = R$ there exists $y \in R$ such that $a + by$ is a (von Neumann) regular element of $R$.
\end{definition}

Obviously, an example of ring (von Neumann) regular range 1 is a ring of stable range 1. Moreover, we have the following result.

\begin{proposition}\label{p2.1.2}
     A commutative ring of (von Neumann) regular range 1 is a ring of stable range 1.
\end{proposition}

\begin{proof}
      Let $R$ be a ring of (von Neumann) regular range 1 and $aR + bR = R$. Then there exists an element $y \in R$ such that $a + by = r$ is a (von Neumann) regular element of $R$. By Proposition \ref{p1.2.2}, we have $a + by = r = ek$, where $e \in \BR$ and $k \in \UR$.

      Note that, since  $aR + bR = R$, we have  $eR + bR = R$. Then  $eu + bv = 1$ for some elements $u, v \in R$. Since $1 - e = (1 - e)eu + (1 - e)bv$, we have $1 - e = (1 - e)bv$, and $e + b(1 - e)v = 1$. Then $ek + b(1 - e)kv = k$.

       Thus, we have $a + bs = k$ for some element $s \in R$, i.e. $(a + bs)R = R$. We have that $R$ is a ring of stable range 1.
\end{proof}

Then we have the following  result.

\begin{theorem}\label{t2.1.3}
      For a commutative ring the following conditions are equivalent:

            1. \; $R$ is a ring of stable range 1;

            2. \; $R$ is a ring of (von Neumann) regular range 1.
\end{theorem}

\begin{definition}\label{d2.1.4}
      A ring $R$ is said to have a semihereditary range 1, if for any $a, b \in R$ such that $aR + bR = R$ there exists $y \in R$ such that $a + by$ is a semihereditary right element of $R$.
\end{definition}

Obviously, an  example of a ring of semihereditary range 1 is a ring of stable range 1 and a commutative semihereditary ring.

A special place in the class of rings of  semihereditary range 1 is taken by semihereditary local rings.

\begin{definition}\label{d2.1.5}
      A commutative ring $R$ is a semihereditary local ring if for any $a, b \in R$ such that $aR + bR = R$ either $a$ or $b$ is a semihereditary element of $R$.
\end{definition}

Obviously, an example of a semihereditary local ring is a (von Neumann) regular local ring and a semihereditary ring. A commutative domain (which is not a local ring) is a semihereditary local ring which is not a (von Neumann) regular local ring.

\begin{proposition}\label{p2.1.6}
     A commutative semihereditary local ring is a ring of semihereditary range 1.
\end{proposition}

\begin{proof}
      Let $R$ be a commutative semihereditary local ring and $aR + bR = R$. If $a$ is a semihereditary element, the representation $a + b0$ is as tequired. It $a$ is not  semihereditary, by condition $aR + (a + b)R = R$, the element $a + b1$ is  semihereditary.
\end{proof}

The  ring $\Z_{36}$ is not a semihereditary local ring, but $\Z_{36}$ is a ring of semihereditary range 1 (see \cite{Cont}).

\begin{definition}\label{d2.1.7}
      A ring $R$ is said to have regular range 1 if for any $a, b \in R$ such that $aR + bR = R$ there exists $y \in R$ such that $a + by$ is a regular element of $R$.
\end{definition}

\begin{theorem}\label{p2.1.8}
      For a commutative ring $R$ the following conditions are equivalent

            1) \; $R$ is a ring of regular range 1;

            2) \; $R$ is a ring of semihereditary range 1.
\end{theorem}


\begin{proof}
      A regular element is a semihereditary element and then if $R$ is a ring of regular range 1 then $R$ is a ring of semihereditary range 1.

      Let $R$ be a ring of semihereditary range 1 and  $aR + bR = R$. Then there exists $y \in R$ such that $a + by=er$, where $e \in \BR$, $r \in \RR$. Since $aR + bR = R$, we have $eR+ bR = R$. Then $eu + bv = 1$ for some elements $u, v \in R$. Since $1 - e = (1 - e)eu + (1 - e)bv$ we have $e + b(1 - e)v = 1$ and $er + br(1 - e)v = r$. Since $a + by = er$ and $er + br(1 - e)v = r$, we have $a + bs = r$ for some element $s \in R$. Then $R$ is a ring of regular range 1.
\end{proof}

\begin{proposition}\label{p2.1.9}
     A classical ring of quotients $\QR$ of a commutative Bezout ring $R$ of regular range 1 is a ring of stable range 1.
\end{proposition}

\begin{proof}
      Let
      $$\frac{a}{s}\QR + \frac{b}{s}\QR = \QR.$$
      Then $au + bv = t$. where $u, v \in R$ and $t \in \RR$. Since $R$ is a commutative Bezout ring, we have $aR + bR = dR$ for some element $d \in R$. Then $a = a_0d$, $b = b_0d$ and
      $ax + by = d$ for some elements $a_0, b_0, x, y \in R$. Since $au + bv = t$, we have
      $d(a_0u + b_0v) = t$. Then $d$ is a regular element as the divisor of a regular element $t$.

      Since $d(a_0x + b_0y - 1) = 0$ and $d \neq 0$, we have $a_0x + b_0y - 1 = 0$ i.e. $a_0R + b_0R = R$. Since $R$ is a ring of regular range 1, we have $a_0 + b_0k = r$ regular element of $R$ for some element $k \in R$. Then $a + bk = rd \in \RR$. So we have $\frac{a}{s} + \frac{b}{s}k = \frac{rd}{s}$.

      Since $\frac{rd}{s} \in \UQR$ we have $(\frac{a}{s} + \frac{b}{s}k)\QR = \QR$ i.e $\QR$ is a ring of stable range 1.
\end{proof}

Here are some examples of  rings of regular range 1.

\begin{definition}\label{d2.1.10}
      A commutative ring $R$ is a regular local ring if for any $a \in R$ either $a$ or $1 - a$ is a regular element.
\end{definition}

\begin{proposition}\label{p2.1.11}
     A commutative regular local Bezout ring is a ring of stable range 2.
\end{proposition}

\begin{proof}
      Let $R$ be a regular local Bezout ring. Let $a, b$ by nonzero elements of $R$. Since $R$ is a commutative Bezout ring, we have $aR + bR = dR$. Then we have $au + bv = d$, $a = a_0d$, $b = b_0d$ for some elements $a_0, b_0, u, v \in R$. Since $d(a_0u + b_0v - 1) = 0$, by the definition of a ring $R$ we see that either $a_0u + b_0v$ or $a_0u + b_0v - 1$ is a regular element of $R$. If $a_0u + b_0v - 1$ is a regular element, by $d(a_0u + b_0v - 1) = 0$ we have $d = 0$, i.e. $a = b =0$ and this  is impossible. Let $a_0u + b_0v = r$ be a regular element of $R$.

      Let $a_0R + b_0R = \de R$. If $\de \notin \UR$ we have $a_0x + b_0y = \de$, $a_0 = \de a_1$, $b_0 = \de b_1$ for some elements $a_1, b_1, x, y \in R$. This implies $\de (a_1u + b_1v) = a_0u + b_0v = r$. Since $r \in \RR$, we have $\de \in \RR$.

      This implies  $\de (a_1x + b_1y - 1) = 0$ and, since $\de \neq 0$, we have $a_1x + b_1y - 1 = 0$ i.e. $a_1R + b_1R = R$. Thus, we have $a = d\de a_1$, $b = d\de b_1$, $a_1R + b_1R = R$. By \cite{GH}, $R$ is an Hermite ring and, by Theorem \ref{t1.1.1}, we obtain that $R$ is a ring of stable range 2.
\end{proof}

In the class of rings of regular range 1 allocate of a class of ring of idempotent regular range 1.

\begin{proposition}\label{p2.1.12}
     A ring $R$ is said to be a ring of idempotent regular range 1 if for any element $a, b \in R$ such that $aR + bR = R$ there exists an idempotent $e \in \BR$ and a regular element $r \in \RR$ such that $a + be = r$.
\end{proposition}

An obvious example of a ring of idempotent regular range 1 is a ring of idempotent stable range 1, i.e a commutative clean ring.

\begin{proposition}\label{p2.1.13}
     A commutative regular local ring is a ring of idempotent regular range 1.
\end{proposition}

\begin{proof}
      Let $R$ be a regular local ring and $aR + bR = R$. If $a$ is a regular element, then  we have a representation $a + b0 = a$. If $a$ is not a regular element, since $aR + (a + b)R = R$, the element $a + b1$ is regular.
\end{proof}

\begin{theorem}\label{t2.1.14}
      A commutative semihereditary ring is a ring of idempotent regular range 1.
\end{theorem}

\begin{proof}
      Let $R$ be a commutative semihereditary ring and $aR+bR=R$. By \cite{McG} and Proposition \ref{p1.2.3}, we have $a=er$ where $e$ is an idempotent and $r$ is a regular element. Note if $e=1$, we have that $a$ is a regular element and $a+b\cdot 0$ is a necessary representation. If $e\neq 1$, let $s=a+b(1-e)$. Show that $s$ is a regular element of $R$. Let $sx=0$, then $ax=-b(1-e)x$. Since $a=er$, we have
      $$erx=(1-e)(-b)x.$$

      Thus, we have $e\cdot erx=e(1-e)(-b)=0$. Since $erx=exr=0$ and $r$ is a regular nonzero element, we have  $ex=0$ and $b(1-e)x=0$, therefore $bx=bex=0$. Hence we have $ax=0$ and $bx=0$. Since $aR+bR=R$ we have $au+bv=1$ for some elements $u,v\in R$. Then $x=axu+bxv=0$ and $s=a+b(1-e)$ is a regular element. Thus, we have that $R$ is a ring of idempotent regular range 1.
\end{proof}

Consequently, we have

\begin{proposition}\label{p2.1.15}
      A commutative ring of idempotent regular range 1 is an almost clean ring.
\end{proposition}

\begin{proof}
      Let $R$ be a ring of idempotent regular range 1 and let $a\in R$ be any nonzero element $a\in R$. Then $aR+(-1)R=R$ and $a-e=r$, where $e$ is an idempotent and $r$ is a regular element of $R$.
\end{proof}

{\bf Open question:}
Is every  commutative almost clean ring a ring of idempotent regular range 1?

\begin{proposition}\label{p2.1.16}
      For a commutative ring $R$ the following conditions are equivalent:

            1) \; $R$ is an  indecomposable almost clean ring;

            2) \; $R$ is a regular local ring.
\end{proposition}

\begin{proof}
      Let $R$ be an indecomposable almost clean ring. Since 0 and 1 are all idempotents of $R$, we have for any a that either $a$ or $1-a$ is a regular element of $R$.

      Let $R$ be a regular local ring. Since for each idempotent $e\in R$ we have, that both $e$ and $1-e$ are  idempotents we have that $R$ is indecomposable ring. By Proposition \ref{p2.1.13}, we have that $R$ is a ring of idempotent regular range 1 and by Proposition \ref{p2.1.15}, $R$ is an almost clean ring.
\end{proof}

By Theorem \ref{t1.1.1} and Proposition \ref{p2.1.11} we have the following result.

\begin{theorem}\label{t2.1.17}
      A commutative indecomposable almost clean Bezout ring is a Hermite ring.
\end{theorem}

\begin{proposition}\label{p2.1.18}
      A commutative semihereditary local ring is a ring of idempotent regular range 1.
\end{proposition}

\begin{proof}
      Let $R$ be a commutative semihereditary local ring and $aR+bR=R$. If $a$ is semihereditary element we have a representation $a=er$, where $e$ is an idempotent and $r$ is a regular element. Then we have that $a+b(1-e)$ is a regular element by the proof of Theorem \ref{t2.1.14}. If $a$ is not a semihereditary element, then by the equality $aR+(a+b)R=R$, we have that $a+b=er$ is a semihereditary element, i.e. $e^{2}=e$ and $r\in \RR$.

      Since $(a+b)R+(-b)R=R$, the equalities $a+b - b(1-e)=a+be=s$ we provide a necessary representation.
\end{proof}

\begin{theorem}\label{t1.2.19}
      Let $R$ be a commutative Bezout ring. Then $\QR$ is a (von Neumann) regular local ring if and only if $R$ is a semihereditary local ring.
\end{theorem}

\begin{proof}
      Let $aR+bR=R$, then $\frac{a}{1}\QR+\frac{b}{1}\QR=\QR$. Since $\QR$ is (von Neumann) regular local ring, either  $\frac{a}{1}$ or $\frac{b}{1}$ is a (von Neumann) regular element. If $\frac{a}{1}$ is a (von Neumann) regular element, then by Proposition \ref{p1.2.2} we have $\frac{a}{1}=eu$, where $e^{2}=e\in \QR$ and $u\in U(\QR)$. By Proposition \ref{p1.2.1}, we have $e\in R$. Then we have $a=er$, where $r$ is a regular element of $R$. The case $\frac{b}{1}$ is a (von Neumann) is similar.

      Let $R$ be a semihereditary local ring and
      $$\frac{a}{s}\QR+\frac{b}{s}\QR=\QR$$
      and either $\frac{a}{s}\neq 0$ or $\frac{b}{s}\neq 0$. Then $au+bv=t$ for some elements $u,v\in R$ and $t$ is a regular element $R$. Since $R$ is a commutative Bezout ring, then $aR+bR=dR$. Let $a=a_{0}d, \; b=b_{0}d$ and $ax+by=d$ for some elements $a_{0},b_{0},x,y\in R$. By the equality $au+bv=t$, we have $d(a_{0}u+b_{0}v)=t$. Then $d$ is a regular element as $a$ divisor of $t$. By the equality $ax+by=d$, we have $d(a_{0}x+b_{0}y-1)=0$. Since $d\neq 0$ and $d$ is a regular element, we have $a_{0}x+b_{0}y=1$. Hence $a_{0}R+b_{0}R=R$ we have $a_{0}$ or $b_{0}$ is a semihereditary element.

      If $a_{0}$ is a semihereditary element, by Proposition \ref{p1.2.3}, we have $a_{0}=er$, where $e^{2}=e$ and $r$ is a regular element of $R$. Since $a=a_{0}d=e(rd)$, we have $\frac{a}{s}=e\frac{rd}{s}$. Since $e^{2}=e$ and $\frac{rd}{s}\in U(\QR)$, we have that $\frac{a}{s}$ is a (von Neumann) regular element. If $b_{0}$ is  (von Neumann) regular, we have a similar proof. Then $\QR)$ is (von Neumann) regular local ring.
\end{proof}

\begin{definition}\label{d2.1.20}\cite{GilHuc}
      A commutative ring $R$ is said to be additely regular if for each $a\in R$ and each regular element $b\in R$ there exists an element $u \in R$ such that $a+ub$ is regular in $R$.
\end{definition}

\begin{proposition}\label{p1.2.21}
      A commutative Bezout ring of regular range 1 is additively regular.
\end{proposition}

\begin{proof}
      Let $R$ be a commutative Bezout ring of regular range 1 and let $a$ be any element $R$ and let $b$ be any regular element of $R$. Since $R$ is a commutative Bezout ring, we have $aR + bR = dR$ and where $au + bv = d$, $a=a_{0}d$, $b=b_{0}d$ for some element $u,v,a_{0},b_{0}\in R$. Since $b$ is a regular element of $R$, we have that $d$ is a regular element of $R$, since $d$ is divisor of $b$.

      Since $au+bv=d$, we have $d(a_{0}u+b_{0}v-1)=0$. Hence $d\neq 0$ and we have $a_{0}u+b_{0}v-1=0$ i.e. $a_{0}R+b_{0}R=R$. Thus, $R$ is  a ring of regular range 1 and we obtain flat  $a_{0}+b_{0}t=r$ is a regular element for some $t\in R$. Then $a+bt=rd$ is a regular ring, i.e. $R$ is an additively regular ring.
\end{proof}

\end{document}